\documentclass[a4paper, 11pt]{article}

\usepackage[english]{babel}
\usepackage[utf8x]{inputenc}
\usepackage[T1]{fontenc}

\usepackage[a4paper,margin=1in]{geometry}

\usepackage[numbib,nottoc]{tocbibind}
\newtheorem{theorem}{Theorem}
\usepackage{amsmath}
\usepackage{indentfirst}
\usepackage{amssymb}
\usepackage{graphicx}
\usepackage{adjustbox}
\usepackage{tabularx}
\usepackage{booktabs}
\usepackage{tabu}
\usepackage{caption}
\usepackage{blindtext}
\usepackage[colorinlistoftodos]{todonotes}
\usepackage[colorlinks=true, allcolors=blue]{hyperref}
\usepackage{listings}
\usepackage{color}
\usepackage{setspace}
\usepackage{apacite}
\usepackage{tikz}
\usetikzlibrary{backgrounds,calc,positioning}
\usetikzlibrary{arrows,decorations.markings}
\usepackage{shapepar}
\usepackage{graphicx}
\usepackage{ragged2e}
\usepackage{listings}
\usepackage{algorithmicx}
\usepackage{algorithm}
\usepackage{pgfplots}
\usepackage{float}
\usepackage{color,soul}
\usepackage{breakurl}

\makeatletter
\def\BState{\State\hskip-\ALG@thistlm}
\makeatother
\usepackage[noend]{algpseudocode}
\usepackage{listings} \usepackage{color} \definecolor{dkgreen}{rgb}{0,0.6,0} \definecolor{gray}{rgb}{0.5,0.5,0.5} \definecolor{mauve}{rgb}{0.58,0,0.82} 

\title{On the Classification and Algorithmic Analysis of \\Carmichael Numbers}
\author{Sathwik Karnik}
\date{}
\begin{document}
\large
\newenvironment{proof}{\paragraph{Proof:}}

\maketitle
\begin{abstract}



In this paper, we study the properties of Carmichael numbers, false positives to several primality tests. We provide a classification for Carmichael numbers with a proportion of Fermat witnesses of less than $50\%$, based on if the smallest prime factor is greater than a determined lower bound. In addition, we conduct a Monte Carlo simulation as part of a probabilistic algorithm to detect if a given composite number is Carmichael. We modify this highly accurate algorithm with a deterministic primality test to create a novel, more efficient algorithm that differentiates between Carmichael numbers and prime numbers. 
\end{abstract}

\clearpage
\tableofcontents
\clearpage
\section{Introduction}


In recent years, cybersecurity has been an issue because of insecure cryptosystems. Primality testing is an important step in the implementation of the RSA cryptosystem. In the search for time-efficient primality tests, composite numbers have been inadvertently selected for key generation, rendering the system fatally vulnerable \cite{pinch_1997}. Carmichael numbers are false positives to several primality tests, including the Fermat test and the Miller-Rabin test \cite{Pinch1993SomePT}. This paper provides both a classification of Carmichael numbers and a novel, highly accurate algorithm that detects Carmichael numbers. 

Section 2 of this paper provides the necessary background for studying the proportion of Fermat witnesses for Carmichael numbers. Furthermore, Section 2 concludes with the observation that many Carmichael numbers have a proportion of Fermat witnesses of less than $50\%.$ 

The results pertaining to the classification of Carmichael numbers with a proportion of Fermat witnesses of less than $50\%$ are detailed in Section 3.1. This classification provides a lower bound for the smallest prime factor of certain Carmichael numbers with a proportion of Fermat witnesses of less than $50\%$ using both inequalities from the initial observation and Newton's method for approximating the root of a function.

The observation made in Section 2.4 served as the motivation for creating an algorithm that differentiates between Carmichael numbers and other composite numbers. Section~3.2 discusses this algorithm, which uses a Monte Carlo simulation to check if a composite number is Carmichael with a certain high probability. The proof of this algorithm and its probability of correctness are detailed in Sections 3.3 and 3.4, respectively. In addition, Section 3.6 provides a modified version of this algorithm that allows for the detection of Carmichael numbers among both composite numbers and prime numbers. The proof of this modified algorithm and its probability of correctness are detailed in Sections 3.7 and 3.8, respectively.

Sections 3.5 and 3.9 analyze the efficiencies of the first algorithm and the modified version. The first algorithm has a run-time of $O(t(\log n)^3)$, where $n$ is the number that is tested and $t$ is the sample size of the number of integers selected in the random sample. The run-time of the second algorithm is $O\left(nt(\log n)^3 + \left(\dfrac{n}{\log n}+C(n)\right)\cdot x \right)$, where $x$ is the run-time of the deterministic primality test that is combined with the original algorithm. Detailed analyses of these efficiencies are provided in Sections 3.5 and 3.9.

\section{Background}

\subsection{Primality Testing}

The RSA algorithm requires two large prime numbers, $p$ and $q,$ from which the keys are generated. To determine if a randomly generated large number $n$ is prime, deterministic primality tests (tests with $100\%$ accuracy) may seem to be the primary option. However, even the fastest known deterministic tests, such as the Agrawal-Kayal-Saxena primality test (or the AKS test), have a run-time of $O((\log n)^6)$, where $n$ is the number that is tested for primality \cite{Klappenecker2002TheAP}. Thus, more efficient primality testing algorithms that maintain a high accuracy are needed. Many practical primality tests for larger numbers are \textit{probabilistic}. In probabilistic primality tests, either (1) a positive integer $n$ is determined to be composite (with $100\%$ accuracy) or (2) the integer $n$ is determined to be prime with a certain probability. To maximize the probability that the primality test works correctly, one must conduct a Monte Carlo simulation so that the chance that $n$ is incorrectly shown to be prime is strictly less than a predetermined value. 

\subsection{Fermat Test}

The Fermat test is a probabilistic primality test that utilizes notions from Fermat’s little theorem \cite{Pinch1993SomePT}. In the Fermat test, a random number $ a $ is chosen from $ (\mathbb{Z}/n \mathbb{Z})\backslash\{0\}$. The test then checks if $ a^{n-1}\equiv 1\pmod{n}$. If $ a^{n-1}\not \equiv 1\pmod{n}$, then $n$ is not a prime number. Otherwise, if $ a^{n-1}\equiv 1 \pmod{n}$, then $ n $ is said to be prime with a certain probability. In particular, there are some composite numbers $n$ for which there exists an $a\in  (\mathbb{Z}/n\mathbb{Z})\backslash\{0\}$ such that $a^{n-1} \equiv 1 \pmod{n}$; one such composite number is $n=561=3\cdot11\cdot17$. In this case, if $a=2$, $a^{n-1}\equiv2^{560}\equiv 1\pmod{561}$. After randomly selecting an element of $(\mathbb{Z}/n\mathbb{Z})\backslash\{0\}$ and calculating $a^{n-1} \pmod{n}$, $ n=561 $ turns out to be a false positive for the Fermat test. One large class of such false positives is Carmichael numbers, which have the property that for all $a\in (\mathbb{Z}/n\mathbb{Z})^{\times}$, $a^{n-1}\equiv 1 \pmod{n}$.




Consider the set of all $a$ in $\{1, 2, 3,\ldots, n-1\}$ for which $a^{n-1} \not\equiv 1 \pmod{n}$. Such values for $a$ are called \textit{Fermat witnesses} for the Fermat primality test because these values of $a$ show that $n$ is not a prime number. Table \ref{tab:fermattest} shows $a^{n-1}\pmod{n}$ for all $a\in (\mathbb{Z}/n\mathbb{Z})\backslash\{0\}$ in the case when $n=21.$ 
\renewcommand{\arraystretch}{1.2}
\begin{table}[H]
  \caption{Fermat Test for $n=21$}
  \label{tab:fermattest}
  
  \centering
  \resizebox{\textwidth}{!}{
  \begin{tabular}{c|c|c|c|c|c|c|c|c|c|c|c|c|c|c|c|c|c|c|c|c}
  \Large $a$ & \Large 1 & \Large {\color{blue} 2} & \Large {\color{blue} 3} & \Large {\color{blue} 4} & \Large {\color{blue}5} & \Large {\color{blue}6} & \Large {\color{blue}7} & \Large 8 & \Large {\color{blue} 9} & \Large {\color{blue}10} & \Large {\color{blue}11} & \Large {\color{blue}12} & \Large 13 & \Large {\color{blue}14} & \Large {\color{blue}15} & \Large {\color{blue}16} & \Large {\color{blue}17} & \Large {\color{blue}18} & \Large {\color{blue}19} & \Large 20 \\
    \hline \Large $a^{n-1}\pmod{n}$ & \Large 1 & \Large 4 & \Large 9 & \Large 16 & \Large 4 & \Large 15 & \Large 7 & \Large 1 & \Large 18 & \Large 16 & \Large 16 & \Large 18 & \Large 1 & \Large 7 & \Large 15 & \Large 4 & \Large 16 & \Large 9 & \Large 4 & \Large 1 \\
  \end{tabular}
  }
\end{table}

In Table \ref{tab:fermattest}, the values of $ a $ that are Fermat witnesses are colored in blue, and for those values, $a^{n-1}\equiv a^{20} \not\equiv 1 \pmod{21}. $ In the $20$-element set $\{1, 2, 3,\ldots, 20\}$, $16$ elements are Fermat witnesses. In other words, for $n=21,$ the proportion of Fermat witnesses is $80\%.$ A number $a$ is defined to be a \textit{non-trivial} Fermat witness if $\gcd(a,n)=1$ and $a^{n-1}\not \equiv1~\pmod{n}.$ Note that $a$ would be considered a \textit{trivial} Fermat witness if $\gcd(a,n)>1$ because $a$ would not be an element of $(\mathbb{Z}/n\mathbb{Z})^{\times}$, which implies that $a^{n-1} \not\equiv 1\pmod{n}$. It has been shown that for $n\in \mathbb{N},$ if there exists a non-trivial Fermat witness, then the proportion of Fermat witnesses is greater than $50\%$ (see Theorem 3.5.4 of \cite{pubkey}). The proof of this claim uses the idea of three disjoint subsets ($A, B,$ and $C$) that categorize all integers in the set $ \{1, 2, 3,\ldots, n-1\} $: \begin{itemize} \item{$A = \{1 \leq a \leq n-1: a^{n-1} \equiv 1\pmod{n} \}$} \item{$ B=\{1\leq a\leq n-1: \gcd(a,n)=1 \text{ and } a^{n-1}\not\equiv 1\pmod{n} \} $} \item{$C=\{1\leq a\leq n-1: \gcd(a,n)>1\}$}\end{itemize}

Composite numbers with no non-trivial Fermat witnesses (equivalently, $|B|=0$) are called \textit{Carmichael numbers}, which are further detailed in Section 2.3.

\subsection{Carmichael numbers}
Carmichael numbers are composite numbers $n$ with the property that for all $a\in \mathbb{N}$ such that $\gcd(a,n)=1,$ $ a^{n-1}\equiv 1\pmod{n}.$ The Fermat test is vulnerable because there are infinitely many Carmichael numbers \cite{alford_granville_pomerance_1994}. 

Carmichael numbers obey Korselt's criterion, which is the equivalent condition to a composite number $n$ being Carmichael \cite{Alford1982ThereAI}. Korselt's criterion states that a composite number $n$ is Carmichael if and only if the following are true:

(i) the number $n$ does not have a square factor greater than $1$

(ii) for all prime factors $p$ of $n$, $(p-1)\vert (n-1). $ 

Suppose $n=(6m+1)(12m+1)(18m+1), $ where $(6m+1)$, $(12m+1)$, and $(18m+1)$ are prime numbers. It is not difficult to show that $6m\vert (n-1),$ $12m\vert (n-1)$, and $18m\vert (n-1) $ because $n-1=(6m+1)(12m+1)(18m+1)-1=1296m^3+396m^2+36m+1-1$=$1296m^3+396m^2+36m$ \cite{dartmouth}. Thus, by Korselt's criterion, such $n$ is Carmichael.

Carmichael numbers are important to study and classify because of their significant role in primality tests. By understanding the importance of Carmichael numbers, cryptographers and number theorists can modify primality tests in a way that Carmichael numbers can be easily identified.

\subsection{Fermat Witnesses for Carmichael Numbers}

Let $a$ be an element of $(\mathbb{Z}/n\mathbb{Z})\backslash\{0\}.$ Recall that $a$ is a Fermat witness for a Carmichael number $n$ if and only if $\gcd(a,n)>1.$ The proportion of Fermat witnesses for Carmichael numbers is an important subject for investigation because it determines the probability that Carmichael numbers will be correctly determined to be composite numbers. Because $\phi(n)=\vert \{a\in (\mathbb{Z}/n\mathbb{Z})\vert \gcd(a,n)=1 \}\vert $, the proportion of Fermat witnesses for Carmichael number is given by $1-\dfrac{\phi(n)}{n-1}.$

It is important to consider a few small examples of the proportion of Fermat witnesses for Carmichael numbers. For the Carmichael number $n=561,$ the proportion of Fermat witnesses is equal to $1-\dfrac{\phi(n)}{n-1}=1-\dfrac{320}{560}\approx0.4286.$ For the Carmichael number $n=1105,$ the proportion of Fermat witnesses is equal to $1-\dfrac{\phi(n)}{n-1}=1-\dfrac{768}{1104}\approx0.3043.$ For the Carmichael number $n=1729,$ the proportion of Fermat witnesses is equal to $1-\dfrac{\phi(n)}{n-1}=1-\dfrac{1296}{1728}\approx0.2504.$

The examples above seem to suggest that the rate of Fermat witnesses is less than 50\% for all Carmichael numbers. However, this conjecture is not correct; Table \ref{tab:greater} lists all Carmichael numbers less than $10^{21}$ with the property that $1-\dfrac{\phi(n)}{n-1}$ is greater than 50\% \cite{Pinch2008TheCN}. Although the rate of Fermat witnesses for Carmichael numbers is not bounded above by $50\%,$ the observations pertaining to the rate of Fermat witnesses for Carmichael numbers are essential to the creation of the algorithms detailed in this paper.

\renewcommand{\arraystretch}{0.7}
\begin{table}[H]
\centering
\caption{Proportion of Fermat Witnesses is Greater Than $50\%$ for Certain Carmichael Numbers}
\label{tab:greater}
\begin{tabular}{>{\centering\arraybackslash}m{1in}|>{\RaggedLeft}p{3.95cm}|p{6.97cm}}
\hline
$1-\dfrac{\phi(n)}{n-1} \text{ } (\%)$ & $\text{Carmichael Number } n$ & Prime factors of $n$\\ \hline
50.04 & 3,852,971,941,960,065 & 3 · 5 · 23 · 89 · 113 · 1409 · 788,129\\

50.10 & 655,510,549,443,465 & 3 · 5 · 23 · 53 · 389 · 2,663 · 34,607\\
50.21 & 13,462,627,333,098,945 & 3 · 5 · 23 · 53 · 197 · 8,009 · 466,649
\\
50.25 & 26,708,253,318,968,145 & 3 · 5 · 17 · 113 · 57,839 · 16,025,297\\
50.76 & 26,904,099,2399,565 & 3 · 5 · 23 · 29 · 4,637 · 5,799,149\\
50.79 & 158,353,658,932,305 & 3 · 5 · 17 · 89 · 149 · 563 · 83,177
\\
50.89 & 1,817,671,359,979,245 & 3 · 5 · 23 · 29 · 359 · 11027 · 45,893\\
51.72 & 16,057,190,782,234,785 & 3 · 5 · 17 · 29 · 269 · 6089 · 1,325,663\\
51.76 & 75,131,642,415,974,145 & 3 · 5 · 23 · 29 · 53 · 617 · 9,857 · 23,297\\
51.95 & 881,715,504,450,705 & 3 · 5 · 17 · 47 · 89 · 113 · 503 · 14,543
\\
52.01 & 31,454,143,858,820,145 & 3 · 5 · 17 · 23 · 2,129 · 39,293 · 64,109
\\
52.13 & 6,128,613,921,672,705 & 3 · 5 · 17 · 23 · 353 · 7,673 · 385,793
\\
52.34 & 12,301,576,752,408,945 & 3 · 5 · 23 · 29 · 53 · 113 · 197 · 1,042,133
\\ 
52.70 & 1,886,616,373,665 & 3 · 5 · 17 · 23 · 83 · 353 · 10,979
\\ 
52.72 & 3,193,231,538,989,185 & 3 · 5 · 17 · 23 · 113 · 167 · 2,927 · 9,857
\\ 
53.26 & 11,947,816,523,586,945 & 3 · 5 ·  17 · 23 · 89 · 113 · 233 · 617 · 1,409
\\
\hline
\end{tabular}
\end{table}

\section{Results}

The properties of Carmichael numbers were used to examine the proportion of Fermat witnesses to find a classification of Carmichael numbers $n$ with the property that the proportion of Fermat witnesses, $1-\dfrac{\phi(n)}{n-1}$ (approximated as $1-\dfrac{\phi(n)}{n}$ for larger values of $n$ in this paper), is less than $50\%.$ Furthermore, this paper provides a novel algorithm that  detects if a given composite number $n$ is Carmichael using observations made about the proportion of Fermat witnesses for Carmichael numbers. In addition, a scheme that combines this highly accurate test with a deterministic primality test is provided to determine if a given number is Carmichael.

\subsection{Classification of Carmichael Numbers $n$ with $1-\dfrac{\phi(n)}{n-1}<50\%$}

Let $ n $ be a Carmichael number such that $ n=p_1p_2\cdots p_r $ and $ p_i $ are all distinct prime factors of $ n$ (it is possible to express a Carmichael number as the product of distinct prime factors by the definition provided in Section 2.3). Let $ a\leq p_1<p_2<\cdots<p_r$. This section focuses on bounding the value of $a$ for which $n$ is guaranteed to be a Carmichael number with $1-\dfrac{\phi(n)}{n-1}<50\%$.

Because there are $ r $ prime factors of $ n $, $ a^r \leq n. $ Using this inequality yields the following: $$ r\log{a}\leq \log {n} $$ $$ r\leq \log_a{n}. $$So, it follows that: $$ \dfrac{1}{a}\geq \dfrac{1}{p_1} $$ $$\left(1-\dfrac{1}{a}\right)^{\log_a{n}}\leq \left(1-\dfrac{1}{a}\right)^{r} \leq \dfrac{\phi(n)}{n}. $$ The last inequality results from the fact that $a$ is less than every prime factor of $n,$ which has $r$ prime factors. Note that $\phi(n)=n\cdot \left(1-\dfrac{1}{p_1}\right)\cdot \left(1-\dfrac{1}{p_2}\right)\cdots \left(1-\dfrac{1}{p_r}\right)\geq \left(1-\dfrac{1}{a}\right)^r$. Furthermore, $\left(1-\dfrac{1}{a}\right)^{\log_a{n}}\leq \dfrac{\phi(n)}{n}$. 

It was observed that many Carmichael numbers have proportions of Fermat witnesses of less than $ 50\%. $ To characterize some Carmichael numbers that exhibit this property, it must now be checked when the following occurs: $$ \dfrac{1}{2}\leq\left(1-\dfrac{1}{a}\right)^{\log_a{n}}\leq \left(1-\dfrac{1}{a}\right)^{r} \leq \dfrac{\phi(n)}{n} $$ $$ \dfrac{1}{2}\leq \left(1-\dfrac{1}{a}\right)^{\log_a{n}} $$ $$ \dfrac{1}{2}\leq \left(\dfrac{a-1}{a}\right)^{\log_a{n}}. $$ Note that $ \log_a n=\dfrac{\log_{(a-1)/a}n}{\log_{(a-1)/a}a},$  which implies that: $$ \dfrac{1}{2}\leq \left(\dfrac{a-1}{a}\right)^{(\log_{(a-1)/a}n)/(\log_{(a-1)/a} a)}=n^{\log_{a}{a-1/a}}. $$ 
Taking the $\log$ of both sides results in: $$ \log \dfrac{1}{2}\leq \left(\log_a{\dfrac{a-1}{a}}\right)\cdot(\log n) $$ $$ \log_n {\dfrac{1}{2}}\leq \log_a{\dfrac{a-1}{a}}. $$ Let $ k=\log_n{\dfrac{1}{2}}. $ Note that $ k\leq\log_{a}{\dfrac{a-1}{a}} $, which implies that $ a^k\leq \dfrac{a-1}{a}. $ Multiplying both sides by $ a $ yields $ a^{k+1} \leq a-1. $ Thus, $ a^{k+1}-a+1\leq 0$. Now, it remains to find the values of $a$ for which $a^{k+1}-a+1\leq 0$.
\\

Let $f(a)=a^{k+1}-a+1$. Figure \ref{figure:a} shows $f(a)$ for the case when $n=1729$. To find the values of $a$ for which $f(a)\leq0$, the zero of $f(a)$ must be calculated. Theorem \ref{one} focuses on this calculation, which results in a classification of Carmichael numbers with a proportion of Fermat witnesses of less than $50\%$. 

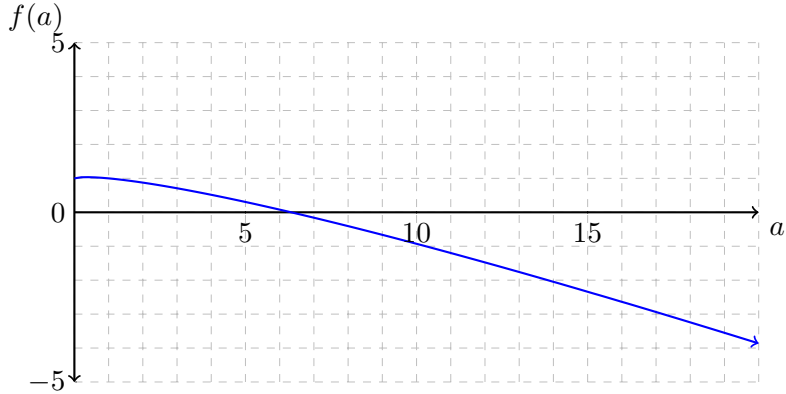
\begin{figure}[H]
\centering
\begin{tikzpicture}[scale=0.45]

\draw [help lines, dashed, color=gray!50, ultra thin] (0,-5) grid [step=1] (20,5);
\draw[thick, ->] (0,0) -- (20,0) node[anchor=north west] {$a$};
\draw[thick, <->] (0,-5) -- (0,5) node[anchor=south east] {$f(a)$};
\draw[thick, blue,->] plot [domain=0:20,samples=100] (\x,{\x^(ln(864.5)/ln(1729))-\x+1});
\foreach \x in {5, 10, 15}
	\draw[thin] (\x, 1pt) -- (\x, -1pt) node[anchor = north] {$\x$};
\foreach \y in {-5,0,5}
	\draw (1pt, \y) -- (1pt,\y) node[anchor=east] {$\y$};
\end{tikzpicture}
\caption{This graph shows the function $f(a)=a^{\log_{n}{n/2}}-a+1$ for $n=1729$.\label{figure:a}}
\end{figure}

\begin{theorem}
\label{one}
If the smallest prime factor $p_1$ of a Carmichael number $n$ satisfies the following: $$ 1+\log_2{n}-\left(\dfrac{(1+\log_2{n})^{\log_n{n/2}}-\log_2{n}}{(\log_n{n/2})\cdot (1+\log_2{n})^{\log_n{1/2}}-1}\right) \leq p_1, $$
then the proportion of numbers from $1$ to $n-1$ that are Fermat witnesses is less than $50\%.$

\end{theorem}
\begin{proof} To find a bound for the zero of $f(a)=a^{k+1}-a+1,$ it suffices to use Newton's method to approximate a lower bound for the smallest prime factor $p_1$ of $n.$ This method begins with a function $f(x)$ defined over the real numbers such that the derivative of $f(x)$ exists and is defined over all reals. An initial guess $x_0$ is made to approximate the root of the function. A new approximation $x_1$ is made using the following equation: $$ x_1=x_0 -\dfrac{f(x_0)}{f'(x_0)}. $$ This process of approximating the roots of the function $f(x)$ continues with: $$ x_{n+1}=x_n - \dfrac{f(x_n)}{f'(x_n)}. $$

Note that the tangents to the function $f(a)=a^{k+1}-a+1$ have $x-$intercepts that are greater than the zero of $f(a)$ because $f(a)$ is a concave function. Thus, if the approximation of the zero of $f(a)$ is less than $p_1$, then the zero of $f(a)$ is less than $p_1$, which implies that the proportion of Fermat witness is less than 50\% for the Carmichael number.

To first approximate the zero of $f(a),$ let $x_0=1.$ Note that $f'(a)=(k+1)\cdot a^k-1,$ which means that $f'(1)=(k+1)\cdot 1 -1=k.$ Also, note that $f(1)=1^{k+1}-1+1=1.$ Thus, $$x_1=1-\dfrac{f(1)}{f'(1)}=1-\dfrac{1}{k}=1+\dfrac{1}{\log_n{2}}=1+\log_2{n}.$$ 

Now, consider the second iteration of Newton's method. Note that: $$x_2=x_1-\dfrac{f(x_1)}{f'(x_1)}=(1+\log_2{n})-\dfrac{f(1+\log_2{n})}{f'(1+\log_2{n})}.$$ 
Furthermore, the numerator of $\dfrac{f(1+\log_2{n})}{f'(1+\log_2{n})}$ can be rewritten as: $$f(1+\log_2{n})=(1+\log_2{n})^{\log_n{(n/2)}}-(1+\log_2{n})+1=(1+\log_2{n})^{\log_n{(n/2)}}-\log_2{n}.$$
The derivative of $f(a)$ evaluated at $a=1+\log_2{n}$ is given by: $$f'(1+\log_2{n})=\left(1+\log_n{(1/2)}\right)\cdot (1+\log_2{n})^{\log_n{(1/2)}}-1.$$
Thus, if the following is true: $$ a<1+\log_2{n}-\left(\dfrac{(1+\log_2{n})^{\log_n{(n/2)}}-\log_2{n}}{(\log_n{(n/2)})\cdot (1+\log_2{n})^{\log_n{(1/2)}}-1}\right)\leq p_1, $$ then the proportion of Fermat witnesses for the Carmichael number $n$ is less than $50\%,$ as desired.
\begin{flushright}
$\square$
\end{flushright}
\end{proof}

\textbf{Theorem \ref{one}} exploits an interesting observation about Carmichael numbers: the proportion of Fermat witnesses for many Carmichael numbers is less than $50\%.$ This property is quite fascinating because every composite number with non-trivial Fermat witnesses has a proportion of Fermat witnesses of greater than $50\%.$ This key observation can be further utilized to create an algorithm that distinguishes between Carmichael numbers and other composite numbers. 
\subsection{Algorithm that Distinguishes Carmichael Numbers and Other Composite Numbers}

This section provides the details for the probabilistic algorithm that determines if a composite number is Carmichael. 

The algorithm works as follows. Consider a composite number $n.$ Conduct a Monte Carlo simulation by first randomly selecting $t$ numbers from the set $\{1, 2,\ldots, n-1\},$ where $t=\lfloor{(\ln n)^2\rfloor}$. Note that $\lfloor{(\ln n)^2\rfloor}$ is the sample size temporarily because $\lfloor{\ln n\rfloor}$ is quite small for larger values of $n$ and the variation would be quite significant with a smaller sample size. For larger numbers, a sample size of $\lfloor{(\ln n)^2\rfloor}$ yields more accurate results. Now, check for each such $a$ from the randomly sample if $a^{n-1}\equiv 1\pmod{n}$. Next, calculate the proportion of values of $a$ for which $a^{n-1}\not\equiv 1\pmod{n}$ from the random sample. If the proportion of such numbers is less than $45\%$ \footnote{Note that there are other composite numbers for which the proportions of Fermat witnesses are close to $50\%.$ Such numbers would be incorrectly determined to be Carmichael because of sampling variations.}, then the composite number $n$ is ``probably'' Carmichael. Otherwise, check every instance in which $a^{n-1}\not\equiv 1\pmod{n}$ and check if $\gcd(a,n)=1.$ If $\gcd(a,n)=1$, then the number $n$ is declared as an ``other composite number.'' If there are no such $a$ relatively prime to $n$, then the number $n$ is Carmichael with a high accuracy. 

The pseudocode for this algorithm is detailed in Algorithm \ref{algorithm:one}.
\renewcommand{\arraystretch}{0.6}
\\
\begin{algorithm}[H]
\setstretch{0.8}
\caption{Determine if a Composite Number is a Carmichael\label{algorithm:one}}
\begin{algorithmic}[1]
\Procedure{CarmichaelDetection}{}
\State $n \gets \text{composite number}$
\State $t \gets \text{size}$
\State $\textit{sample} \gets \text{randomly chosen numbers from $1$ to $(n-1)$} $
\State $\text{indicator} \gets 1 \text{ if Fermat witness, else 0}$
\State $\textit{sample(i)} \gets \textit{$i^{th}$ sample}$
\State $k \gets \text{number of non-trivial Fermat witnesses}$
\BState \textit{t = $floor((\ln n)^2)$}
\BState \textit{randomsample(t,n)}
\BState \emph{loop}:
\If {$(sample(i))^{n-1} \equiv 1 \pmod{n}$} 
\State $indicator.append(0)$
\EndIf
\State $ \textbf{else } indicator.append(1)$
\BState \textbf{if} \textit{ sum(indicator)}\text{ < 45\%:}
\Return \text{Carmichael} 
\BState \textbf{else} 
\State \emph{loop:} 
\State \textbf{if } $\gcd(sample(i), n) == 1$ \text{ and } $indicator[i]==1:$ 
\Return \text{Other composite}
\State \textbf{break}
\State \textbf{else if } $i == t-1:$ 
\Return \text{Carmichael}
\BState \textbf{close}
\EndProcedure
\end{algorithmic}
\end{algorithm}

\subsection{Proof of Correctness for Algorithm 1}

If a number $n$ is Carmichael, then $n$ must have no non-trivial Fermat witnesses. Thus, a Carmichael number $n$ will be accurately determined as Carmichael. Otherwise, other composite numbers, which must have non-trivial Fermat witnesses, will be correctly determined as ``other composite numbers'' with a certain high probability, as described in Section 3.4.

\subsection{Justification of Algorithm 1}

To show that Algorithm \ref{algorithm:one} works with high accuracy, one must consider the probability that a number is Carmichael given that the number is composite and has no non-trivial Fermat witnesses for a random sample of $t=\lfloor{(\ln n)^2\rfloor}$ integers from $1$ to $n-1.$ The proof of this algorithm requires Bayes' rule in conditional probability.

Let $X$ be the random variable for the event that a 1024-bit integer $n$ is Carmichael. Let $Y_t$ be the random variable for the event that either the proportion of Fermat witnesses is less than $45\% $ for the random sample of size $t$ or no non-trivial Fermat witnesses are found after checking if $a^{n-1} \equiv 1\pmod{n}$ for each element of the random sample. Also, let $Z$ be the event that a 1024-bit integer $n$ is composite. The desired probability is equivalent to $Pr(X|(Y_t \cap Z)).$

Recall that Bayes' rule states that: $$ Pr(X|(Y_t \cap Z)) =\dfrac{Pr((Y_t \cap Z)|X)\cdot Pr(X)}{Pr((Y_t\cap Z)|X)\cdot Pr(X) +Pr((Y_t\cap Z)|X')\cdot Pr(X')}. $$
Note that $X'$ refers to the event that $n$ is not Carmichael. 

First, consider the numerator of the probability described above. Note that $Pr((Y_t\cap Z)|X)=1$ because if a number is Carmichael, then $Z$ must be true because all Carmichael numbers are composite numbers and $Y_t$ must be true because Carmichael numbers have no non-trivial Fermat witnesses. Thus, the numerator is equal to $Pr(X)$. Finding the probability that a given 1024-bit integer (a common size of the prime numbers chosen for the RSA cryptosystem) is Carmichael is equivalent to finding the proportion of 1024-bit integers that are Carmichael numbers. The probability $Pr(X)$ can also be expressed as $\dfrac{C(2^{1024})-C(2^{1023})}{2^{1023}},$ where $C(n)$ is a function of $n$ that denotes the number of Carmichael numbers less than a number $n$. It has been found that $$C(n) = n\cdot \exp \left(-k(n)\cdot \dfrac{\log n \log\log\log n}{\log \log n}\right)$$ for some function $k(n)$ defined over $\mathbb{R}$ \cite{Pinch2008TheCN}. Note that $\dfrac{C(n)}{n}$ has been shown to be approximately $\dfrac{n^{0.34}}{n}$ for larger values of $n$. Thus, the numerator can be expressed as $Pr(X)=\dfrac{(2^{1024})^{0.34}-(2^{1023})^{0.34}}{2^{1023}}$.

Consider the denominator of the probability of accuracy for Algorithm \ref{algorithm:one}: $$ Pr((Y_t\cap Z)|X)\cdot Pr(X) + Pr((Y_t\cap Z)|X')\cdot Pr(X'). $$ Note that $Pr((Y_t\cap Z)|X)\cdot Pr(X)$ is equal to the numerator, which is simply $Pr(X)$. Now, consider  the term $ Pr((Y_t\cap Z)|X')\cdot Pr(X')$. Recall that $Pr(X')$ denotes the probability that $n$ is not Carmichael. Because $X'$ is the random variable for the event that $n$ is not Carmichael, $Pr((Y_t\cap Z)|X')$ is the probability that either the proportion of Fermat witnesses is less than $45\%$ for the random sample or no non-trivial Fermat witnesses are found and $n$ is composite, given that the number is not Carmichael. Note that: $$Pr((Y_t\cap Z)|X')=p + (1-p)\cdot \left[\left(1-\dfrac{|B|}{n}\right)^t - \left(\dfrac{|A|}{n}\right)^t\right],$$ where $p$ is the probability that less than $45\%$ of the random sample are Fermat witnesses given that $n$ is not Carmichael. Also, recall that $B$ denotes the set of non-trivial Fermat witness and $A$ denotes the set of all Fermat non-witnesses. 

The expression for $Pr((Y_t\cap Z)|X')$ provided in the previous paragraph can be explained by the intuition behind Algorithm~\ref{algorithm:one}. In this algorithm, Carmichael numbers are first detected based on whether or not the proportion of Fermat witnesses is less than $45\%$. If the proportion of Fermat witnesses is greater than or equal to $45\%,$ then the algorithm checks if there are any non-trivial Fermat witnesses. Similarly, in calculating the probability $Pr((Y_t\cap Z)|X'),$ one must first account for the event that the proportion of Fermat witnesses is less than $45\%$ for the sample. This first part is denoted by $p$, as defined earlier. Otherwise, if the proportion of Fermat witnesses for the sample is greater than or equal to $45\%$, then the probability is given by $(1-p)\cdot \left[\left(1-\dfrac{|B|}{n}\right)^t - \left(\dfrac{|A|}{n}\right)^t\right].$ This is because the probability that the proportion of Fermat witnesses for the sample is greater than or equal to $45\%$ is $(1-p)$ and the probability that there are no non-trivial Fermat witnesses but there are some trivial Fermat witnesses found in the sample is $\left[\left(1-\dfrac{|B|}{n}\right)^t - \left(\dfrac{|A|}{n}\right)^t\right]$ (in the case that there are no Fermat witnesses found, the number $n$ could be prime, which would violate the event $X$).

The equivalent expression for $Pr((Y_t\cap Z)|X')$ described earlier can be evaluated by first approximating the value of $p.$ The distribution of proportions of Fermat witnesses for the random samples is a binomial distribution with an average value of $1-\dfrac{|A|}{n}$ because $A$ denotes the set of all Fermat non-witnesses. Because the proportions of Fermat witnesses from random samples follow a binomial distribution, the standard deviation is given by $ \sigma =\sqrt{\dfrac{1}{t}\left(\dfrac{|A|}{n}\right)\cdot \left(1-\dfrac{|A|}{n}\right)}. $ Since the RSA cryptosystem selects two large prime factors (of about 300 digits), the binomial distribution can be approximated by the probability density function, which describes a normal model (see Figure \ref{figure:bell}).

\pgfmathdeclarefunction{gauss}{2}{%
  \pgfmathparse{1/(#2*sqrt(2*pi))*exp(-((x-#1)^2)/(2*#2^2))}%
}

\pgfplotsset{tick label style={color=white},
  label style={font=\small},
  legend style={font=\small}
 }

\begin{figure}[H]
\centering
\begin{tikzpicture}[scale=1]
\begin{axis}[
  no markers, domain=0.1:0.4, samples=414,
  axis lines*=left, xlabel=$x$, ylabel=$y$,
  every axis y label/.style={at=(current axis.above origin),anchor=south},
  every axis x label/.style={at=(current axis.right of origin),anchor=west},
  height=5cm, width=12cm,
  xtick={0.246}, ytick=\empty,
  enlargelimits=false, clip=false, axis on top,
  grid = major
  ]
  \addplot [very thick,cyan!50!black] {gauss(0.246,0.021156)};
  
\end{axis}

\node at (5.1,-0.6) {$ 1-\dfrac{|A|}{n} $};

\node at (0.8,-0.6) {\tiny $ \sigma =\sqrt{\dfrac{1}{t}\left(\dfrac{|A|}{n}\right)\cdot \left(1-\dfrac{|A|}{n}\right)} $};
\end{tikzpicture}
\caption{Distribution of Fermat witnesses for the Random Sample from $ (\mathbb{Z}/n\mathbb{Z})\backslash\{0\} $\label{figure:bell}}
\end{figure}
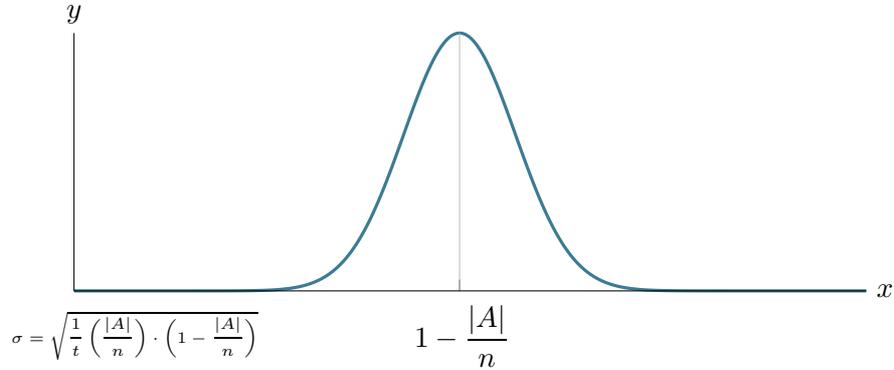

It suffices to find an approximate value of $p,$ which may be found by approximating the value of $\sigma$ and finding the value of $x$ for which the lower values of $x$ represent the event that the proportion of Fermat witnesses is less than $45\%$ for the random sample found by a Monte Carlo simulation. Recall that the proportion of Fermat witnesses for all composite numbers with non-trivial Fermat witnesses is greater than $50\%.$ In other words, $\dfrac{|A|}{n}<\dfrac{1}{2}, $ which implies that $ \left(1-\dfrac{|A|}{n}\right)>\dfrac{1}{2},$ because there exists at least one non-trivial Fermat witness when determining $Pr((Y_t\cup Z)|X)$. To prove that Algorithm \ref{algorithm:one} works for approximately $100\%$ of the time, it suffices to show that when the $1-\dfrac{|A|}{n}=\dfrac{1}{2}$ this accuracy still holds.\footnote{Note that it follows from Lagrange's theorem that the group of all Fermat non-witnesses divides the order of the group $(\mathbb{Z}/n\mathbb{Z})^{\times}$. The least proportion of Fermat witnesses for a number $n$ with non-trivial Fermat witnesses is $1-\dfrac{\phi(n)}{2n}$ because of numbers such as $91$ that can be expressed as $q\cdot (2q-1),$ where $q$ and $2q-1$ are prime. For 91, $q=7.$} 

To find the probability $p,$ one must calculate the number of standard deviations $x=0.45$ is from the mean of $\dfrac{1}{2}$ (this value is also referred to as a $z-$score or standard score): $$ z=\dfrac{0.45-\dfrac{1}{2}}{\sqrt{\dfrac{1}{t}\left(\dfrac{1}{2}\right)\cdot \left(\dfrac{1}{2}\right)}}.$$

Recall that $p$ is the probability that less than $45\%$ of the random sample are Fermat witnesses given that the number $n$ is not Carmichael. The value of $p$ is also equal to the area under the probability density function from $-\infty$ to $z.$ This area can be calculated using the cumulative distribution function, $F(x):$ 



$$ F(z)=\dfrac{1}{\sqrt{2\pi}}\int_{-\infty}^{z}e^{-t^2/2}dt, $$ where $z$ is the standard score.  

For the calculation of the value of $p$ using the cumulative distribution function, a program in \textit{Mathematica} can be used to approximate the value of $z,$ which can then be used to evaluate $F(z).$ To calculate the probability that Algorithm \ref{algorithm:one} works correctly, one may use the approximate size ($\approx 10^{300}$) of the prime numbers used in the RSA cryptosystem to approximate the value of $n.$ In particular, the calculation of the probability depends only on the size of the number $n$ and not on actual prime factors of the number $n$. The calculated probability $p$ yields a probability of approximately $100\%.$ 

\subsection{Efficiency of Algorithm 1}

Using the Algorithm \ref{algorithm:one} implementation and the Algorithm \ref{algorithm:one} pseudocode, it can be calculated that Algorithm \ref{algorithm:one} has a time complexity of $O(t(\log n)^3)$, where $t$ is the sample size. The $(\log n)^3$ represents time needed for determining the greatest common divisor of an element of the sample and $n$ using the Euclidean algorithm. Although this algorithm maintains both high efficiency and high accuracy, Algorithm \ref{algorithm:one} may not be compared to previous primality testing algorithms or previous Carmichael detecting algorithms because it relies on the fact that the number $n$ is composite. Thus, to compare this algorithm with existing algorithms, modifications must be made in a way that Carmichael numbers are detected among not just composite numbers but all numbers.

\subsection{Algorithm 1 Modifications: Detecting Carmichael Numbers}

In Section 3.2, a novel algorithm for distinguishing Carmichael numbers and other composite numbers was described. This algorithm combined the properties of the Fermat witnesses for Carmichael numbers and other fundamental properties. This section exploits the aforementioned scheme to show a new algorithm that allows for the detection of Carmichael numbers and not just the separation between Carmichael numbers and other composite numbers.

Instead of differentiating between Carmichael numbers and other composite numbers, one may modify the algorithm so that it could differentiate between the set of both Carmichael numbers and prime numbers and the set of all other composite numbers. This modification allows for a deterministic (or almost deterministic) primality test to check all of the numbers in the set of all Carmichael numbers and prime numbers, which is much smaller to check than the set of all integers. 

\subsection{Proof of Correctness for Modified Algorithm}

If a number $n$ is Carmichael or prime, then $n$ must have no non-trivial Fermat witnesses. Otherwise, other composite numbers, which must have non-trivial Fermat witnesses, will be correctly determined to be ``other composite numbers'' with a certain high probability, as described in Section 3.8. Furthermore, a highly accurate primality test that has been proven for correctness will correctly distinguish between Carmichael numbers and prime numbers.

\subsection{Justification of the Modified Algorithm}

This section provides a proof for the high accuracy of the modified algorithm. The proof detailed in this section uses similar notions as those used in Section 3.4. However, the random variable for the event that the number is composite will not be of use in this proof that justifies the distinction of Carmichael numbers among all other integers. 

Let $X$ be the random variable for the event that a 1024-bit integer is either Carmichael or prime. Also, let $Y_t$ represent the random variable for the event that after random sampling $t=\lfloor{(\ln n)^2\rfloor}$ times, either the proportion of Fermat witnesses for a number $n$ is less than $45\%$ or there are no non-trivial Fermat witnesses. The probability that must be calculated is as follows: $$ Pr(X|Y_t)=\dfrac{Pr(Y_t|X)\cdot Pr(X)}{Pr(Y_t|X)\cdot Pr(X)+Pr(Y_t|X')\cdot Pr(X')}. $$

Note that $Pr(Y_t|X)=1$ because Carmichael numbers and prime numbers have no non-trivial Fermat witnesses. So, the numerator is equal to $Pr(X),$ which is the probability that a randomly chosen number is Carmichael or composite. Calculating this probability is the same as calculating the proportion of numbers less than a number $n$ that are Carmichael or prime. As detailed in Section 3.4, the proportion of numbers that are Carmichael is approximately $\dfrac{n^{0.34}}{n}.$ The proportion of numbers less than $n$ that are prime is approximately $\dfrac{1}{\ln n},$ which is a result of the prime number theorem. Thus, accounting for the size of the prime numbers used in the RSA cryptosystem, it may be calculated that $Pr(X)=\dfrac{(2^{1024})^{0.34}-(2^{1023})^{0.34}}{2^{1023}}+\dfrac{\dfrac{2^{1024}}{\ln 2^{1024}}-\dfrac{2^{1023}}{\ln 2^{1023}}}{2^{1023}}.$

For the denominator, $Pr(Y_t|X)\cdot Pr(X)$ is the same as the numerator. Now, consider the term $Pr(Y_t|X')\cdot Pr(X').$ The left term, $Pr(Y_t|X'),$ represents the probability that a number that is neither Carmichael nor prime has either a proportion of Fermat witnesses that is less than $45\%$ or no non-trivial Fermat witnesses. This probability is exactly the same as $Pr(Y_t\cap Z|X') = p + (1-p)\cdot \left[\left(1-\dfrac{|B|}{n}\right)^t - \left(\dfrac{|A|}{n}\right)^t\right],$ as shown in Section 3.4. Thus, the probability $Pr(X|Y_t)$ is equal to: $$ Pr(X|Y_t)=\dfrac{Pr(X)}{Pr(X) + (1-Pr(X))\cdot \left[p + (1-p)\cdot \left[\left(1-\dfrac{|B|}{n}\right)^t - \left(\dfrac{|A|}{n}\right)^t\right]\right]}, $$ which can be evaluated using \textit{Mathematica} to approximate the probability using large numbers for $n.$ Thus, the probability of accuracy of the modified algorithm is approximately $100\%.$

\subsection{Efficiency of the Modified Algorithm}

The modified algorithm is useful for finding a list of Carmichael numbers less than or equal to $n$. Suppose that the algorithm runs for the first $n$ numbers. Then, the time complexity of the modified algorithm is $O\left(nt(\log n)^3 + \left(\dfrac{n}{\log n}+C(n)\right)\cdot x \right),$ where $x$ is the run-time of a deterministic primality test that is combined with Algorithm~\ref{algorithm:one}. Note that $ \left(\dfrac{n}{\log n}+C(n)\right)\cdot x $ represents the time needed for the modified part of the algorithm. Although this modified algorithm is efficient for the purposes of determining a list of Carmichael numbers, the efficiency could be optimized by finding a value of $t$ for which the accuracy is still maintained.

\section{Conclusions and Future Extensions}

This paper determined both a classification of Carmichael numbers and a method for detecting Carmichael numbers, pseudoprimes to several primality tests. To further the research in this paper, one may examine the proportion of Fermat witnesses to find the percentage of Carmichael numbers with a proportion of Fermat witnesses of less than $50\%.$ These findings may be used to modify the upper bound for which the proportion of Fermat witnesses is checked in Algorithm \ref{algorithm:one}. Furthermore, the algorithm may be modified with an efficient deterministic primality test. Moreover, the value of $t$ must be modified to improve the efficiency of the algorithm. To extend the idea of detecting pseudoprimes, one may examine either the proportion of witnesses for false positives of other primality tests that have many false positives. 

\section{Acknowledgments}

The author wishes to thank his mentor, Hyun Jong Kim, for his guidance throughout this project. The author would also like to thank Dr.~Tanya Khovanova for helping to edit this paper and the MIT PRIMES program for making this research possible. 

\clearpage
\bibliographystyle{apacite}
\bibliography{references.bib}





\end{document}